\documentclass[12pt,reqno]{amsart}
\usepackage{fullpage}
\usepackage{amsfonts}
\usepackage{enumerate}
\usepackage{times}
\usepackage{graphicx}
\usepackage{url}

\renewcommand{\mod}[1]{{\ifmmode\text{\rm\ (mod~$#1$)}\else\discretionary{}{}{\hbox{ }}\rm(mod~$#1$)\fi}}

\newcommand{\ep}{\varepsilon}
\newcommand{\tr}{\mathop{\mathrm{tr}}}
\newcommand{\disc}{\mathop{\mathrm{disc}}}

\newcommand{\one}[1]{\mathbf{1}\{#1\}}

\newcommand{\M}{\mathcal{M}}
\newcommand{\Ntwo}{\mathcal{M}_2([-1,1])}

\newcommand{\R}{{\mathbb R}}
\newcommand{\Z}{{\mathbb Z}}
\renewcommand{\P}{{\rm Pr}}

\newtheorem{theorem}{Theorem}
\newtheorem{proposition}[theorem]{Proposition}
\newtheorem{corollary}[theorem]{Corollary}
\newtheorem{lemma}[theorem]{Lemma}
{
\theoremstyle{remark}
\newtheorem*{remark}{Remark}
}

\begin{document}

\title{The number of $2\times2$ integer matrices having a prescribed integer~eigenvalue}
\author{Greg Martin and Erick B.\ Wong} % \\ Draft: \today}
\address{Department of Mathematics \\ University of British Columbia \\ Room
121, 1984 Mathematics Road \\ Canada V6T 1Z2}
\email{gerg@math.ubc.ca and erick@math.ubc.ca}
\maketitle

\begin{abstract}
Random matrices arise in many mathematical contexts, and it is natural to ask about the properties that such matrices satisfy.  If we choose a matrix with integer entries at random, for example, what is the probability that it will have a particular integer as an eigenvalue, or an integer eigenvalue at all? If we choose a matrix with real entries at random, what is the probability that it will have a real eigenvalue in a particular interval? The purpose of this paper is to resolve these questions, once they are made suitably precise, in the setting of $2\times2$ matrices.
\end{abstract}

\section{Introduction}

Random matrices arise in many mathematical contexts, and it is natural to ask about the properties that such matrices satisfy. If we choose a matrix with integer entries at random, for example, we would like to know the probability that it has a particular integer as an eigenvalue, or an integer eigenvalue at all. Similarly, if we choose a matrix with real entries at random, we would like to know
the probability that it has a real eigenvalue in a particular interval. Certainly the 
answer depends on the probability distribution from which the matrix entries are drawn.

In this paper, we are primarily concerned with uniform distribution, so for both integer-valued and real-valued cases we must restrict the entries to a bounded interval.  In an earlier paper \cite{MW}, the authors show that random $n\times n$ matrices of integers almost never have integer eigenvalues. An explicit calculation by Hetzel, Liew, and Morrison \cite{HLM} shows that a $2\times2$ matrix with entries independently chosen uniformly from $[-1,1]$ has real eigenvalues with probability $49/72$. This calculation gives hope that our more precise questions about eigenvalues of a particular size might be accessible in the $2\times2$ setting. The purpose of this paper is to resolve these questions, once we make them suitably precise.

For an integer $k \ge 1$, let $\M_2(k)$ denote the uniform probability space of $2\times2$ matrices of integers with absolute value at most $k$.  Note that $|\M_2(k)| = (2k+1)^4 = 16k^4 + O(k^3)$.  We will obtain exact asymptotics for the number of matrices in $\M_2(k)$ having integer eigenvalues and, more precisely, for the number of matrices with a given integer eigenvalue~$\lambda$.

For any integer $\lambda$, define
\[
\M^\lambda_2(k) = \{M \in \M_2(k)\colon \lambda \text{ is an eigenvalue of } M\},
\]
and let
\[
\M^{\Z}_2(k) = \bigcup_{\lambda \in \Z} \M^{\lambda}_2(k).
\]
%We will prove the following estimate for the cardinality of $\M^{\Z}_2(k)$:

\begin{theorem}
Define the function $V:[-2,2]\to\R$ by $V(-\delta)=V(\delta)$ and
\begin{equation}
V(\delta) = \begin{cases}
4 - 2\delta - \delta^2 + \delta^2 \log(1+\delta) - 2(1-\delta)\log(1-\delta),
& \text{if }0 \le \delta < 1, \\
1+\log2,
& \text{if }\delta=1, \\
4 - 2\delta - \delta^2 + \delta^2 \log(\delta+1) + 2(\delta-1)\log(\delta-1),
& \text{if }1 < \delta \le \sqrt2, \\
\delta^2 - 2\delta - (\delta^2-2\delta+2) \log(\delta-1)
& \text{if }\sqrt{2} < \delta \le 2
\end{cases}
\label{definition of V(delta)}
\end{equation}
(where $\log$ is the natural logarithm). Then for any integer $\lambda$ between $-2k$ and $2k$,
\begin{equation}
|\M^\lambda_2(k)| = \frac{24V(\lambda/k)}{\pi^2} k^2 \log k + O(k^2),
\label{how many matrices have eigenvalue lambda}
\end{equation}
where the implied constant is absolute.  On the other hand, if $|\lambda|>2k$ then $\M^\lambda_2(k)$ is empty.
\label{theorem: particular integer eigenvalue} 
\end{theorem}

We remark that the function $V(\delta)$ is continuous and, with the exception of the points of infinite slope at $\delta=\pm1$, differentiable everywhere (even at $\delta=\pm2$, if we imagine that $V(\delta)$ is defined to be 0 when $|\delta|>2$). Notice that equation~\eqref{how many matrices have eigenvalue lambda} is technically not an asymptotic formula when $\lambda$ is extremely close to $\pm2k$, because then the value of $V(\lambda/k)$ can have order of magnitude $1/\log k$ or smaller, making the ``main term'' no bigger than the error term. However, equation~\eqref{how many matrices have eigenvalue lambda} is truly an asymptotic formula for $|\lambda| < 2k - \psi(k)k/(\log k)^{1/3}$, where $\psi(k)$ is any function tending to infinity (the exponent $1/3$ arises because $V(\delta)$ approaches 0 cubically as $\delta$ tends to~2 from below).

By summing the formula \eqref{how many matrices have eigenvalue lambda} over all possible values of $\lambda$, we obtain an asymptotic formula for $|\M^{\Z}_2(k)|$.  We defer the details of the proof to Section~\ref{enumeration}.
\begin{corollary}
Let $C = \big(7\sqrt{2} + 4 + 3\log(\sqrt{2}+1)\big)/3\pi^2 \approx 0.55873957$.  The probability that a randomly chosen matrix in $\M_2(k)$ has integer eigenvalues is asymptotically $C(\log k)/k$.  More precisely,
\[ |\M^{\Z}_2(k)| = 16Ck^3 \log k + O(k^3). \]
\label{corollary: all integer eigenvalues}
\end{corollary}

If $M \in \M_2(k)$ has eigenvalue $\lambda$, then the scaled matrix $k^{-1} M$ has eigenvalue $\lambda/k$, which is the argument of $V$ that appears on the right-hand side of \eqref{how many matrices have eigenvalue lambda}.  Thus one interpretation of Theorem~\ref{theorem: particular integer eigenvalue} is that for large $k$, the rational eigenvalues of $k^{-1} M$ tend to be distributed like the function $V$.

Note that the entries of $k^{-1} M$ are sampled uniformly from a discrete, evenly-spaced subset of $[-1,1]$.   As $k \to \infty$ this probability distribution converges in law to the uniform distribution on the interval $[-1,1]$.  Let $\Ntwo$ denote the probability space of all $2\times2$ matrices whose entries are independent random variables drawn from this distribution.  One might ask whether the distribution given by Theorem~\ref{theorem: particular integer eigenvalue} is just a discrete approximation to the distribution of eigenvalues in $\Ntwo$; the answer, perhaps surprisingly, is no.  The next theorem provides this latter distribution.

\begin{theorem}
Define $W(\delta)$ to be the density function for real eigenvalues of matrices in $\Ntwo$: if $M$ is a randomly chosen matrix from $\Ntwo$, then the expected number of eigenvalues of $M$ in the interval $[s,t]$ is $\int_s^t W(\delta)\, d\delta$. Then $W(-\delta)=W(\delta)$ and
\begin{equation}
W(\delta) = 
\begin{cases}
(80 + 20\delta + 90\delta^2 + 52\delta^3 - 107\delta^4)/(144(1+\delta))
\\ \qquad{} - (5-7\delta+8\delta^2)(1-\delta) \log(1-\delta)/12
\\ \qquad{} - \delta(1-\delta^2) \log(1+\delta)/4,
  & \text{if $0 \le \delta \le 1$},\\
%15/32,
%& \text{if }\delta=1, \\
\delta(20+10\delta-12\delta^2-3\delta^3)/(16(1+\delta))
\\ \qquad{} + (3\delta-1)(\delta-1)\log(\delta-1)/4
\\ \qquad{} + \delta (\delta^2-1) \log(\delta+1)/4,
  & \text{if $1 \le \delta \le \sqrt{2}$},\\
\delta(\delta-2)(2-6\delta+3\delta^2)/(16(\delta-1))
\\ \qquad{}  - (\delta-1)^3 \log(\delta-1)/4,
  & \text{if $\sqrt{2} \le \delta \le 2$}, \\
0,
& \text{if }\delta\ge2.
\end{cases}
\label{definition of W(delta)}
\end{equation}
\label{theorem: distribution of real eigenvalues}
\end{theorem}

As in the case of $V(\delta)$, the function $W(\delta)$ is continuous and differentiable everywhere, with the exception of the points of infinite slope at $\delta=\pm1$. (The value $W(1)=15/32$ makes the function continuous there, although the value of a density function at a single point is irrelevant to the probability distribution.) It also shares the same cubic decay as $\delta$ tends to 2 from below.  However, there are obvious qualitative differences between the functions $V$ and $W$.  In Figure~\ref{figure} we plot $V$ and $W$ on the same axes, normalized so that the area under each is~2 (these normalized versions are denoted $U^\Z$ and $U^\R$ in our earlier paper \cite{MW}).  In the case of $\M_2(k)$, this normalization corresponds to conditioning on having integer eigenvalues, that is, scaling by the probability $C(\log k)/k$ from Corollary~\ref{corollary: all integer eigenvalues}.  For $\Ntwo$ we are 
conditioning on having real eigenvalues, which occurs with probability $49/72$ (this can be 
obtained by integrating $W(\delta)$, analogously to the proof of Corollary~\ref{corollary: all integer eigenvalues}; the computation by Hetzel, Liew, and Morrison \cite{HLM} is more direct).

\begin{figure}[b]
\centering
\includegraphics[width=4in]{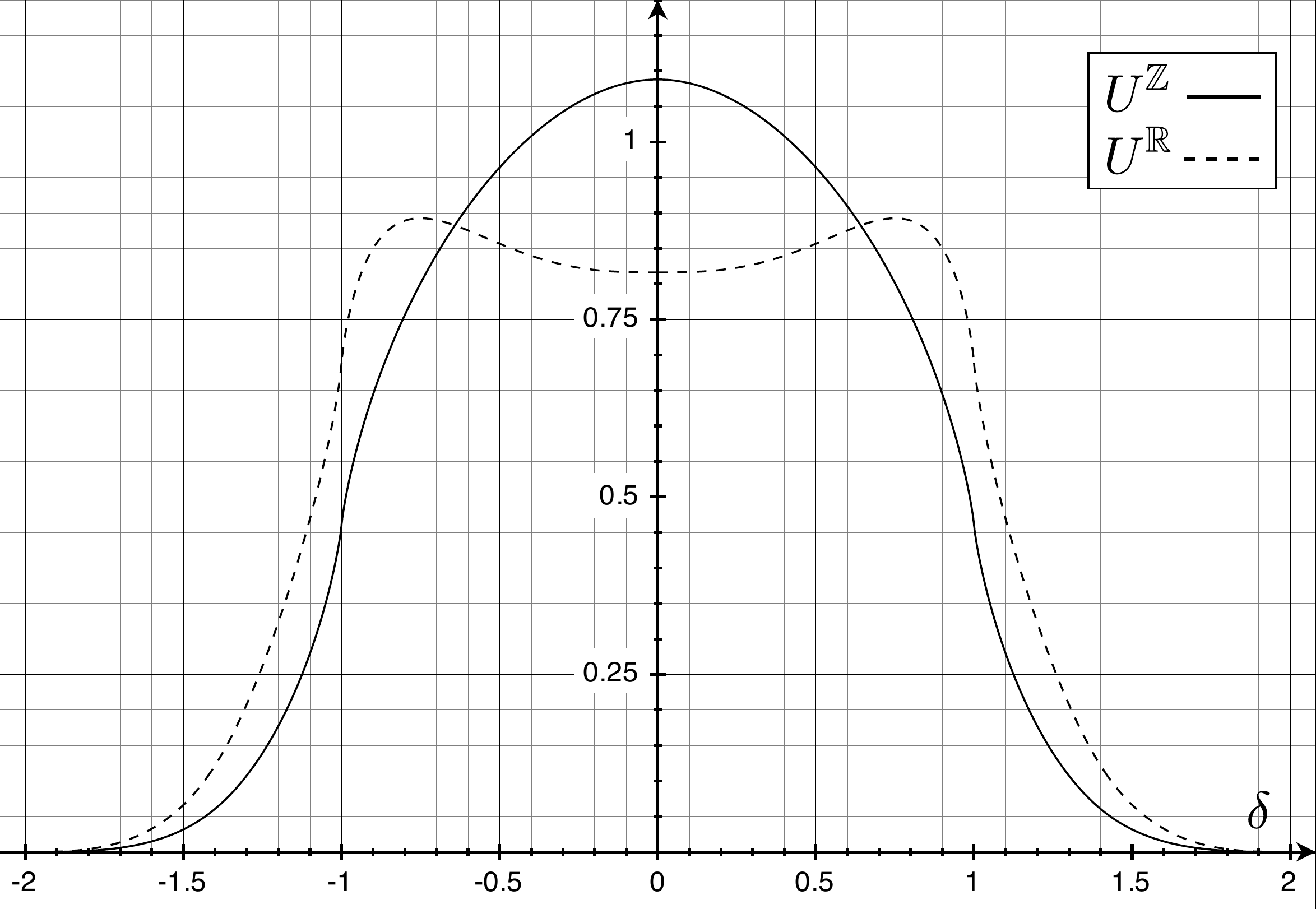}
\caption{Graph of $U^\Z$ ($V$ normalized) versus $U^\R$ ($W$ normalized)}
\label{figure}
\end{figure}

Note that the distribution $W(\delta)$ is bimodal, having its maxima at $\delta \approx \pm 0.75030751.$  Thus, a random matrix in $\Ntwo$ is more likely to have an eigenvalue of magnitude near $3/4$ than one of magnitude near $0$.  We expect this would still hold if we were to condition on matrices in $\Ntwo$ having rational eigenvalues, since any matrix with real eigenvalues is a small perturbation from one with rational eigenvalues.  That this is not true for $V(\delta)$ shows that the eigenvalue distribution of $\M_2(k)$ is not purely the result of magnitude considerations but also encodes some of the arithmetic structure of the integers up to $k$.

We remark that Theorem~\ref{theorem: particular integer eigenvalue} can also be obtained from a powerful result of Katznelson \cite{K}.  Let ${\mathcal B}$ be a convex body containing the origin in $\R^4$, and embed the set of $2\times 2$ integer matrices $\big( \genfrac{}{}{0pt}{1}{a}{b} \genfrac{}{}{0pt}{1}{c}{d} \big)$ as lattice points $(a,b,c,d) \in \Z^4$.  Then Theorem~1 of \cite{K} gives an asymptotic formula for the number of singular integer matrices inside the dilate $k \cdot {\mathcal B}$.  Taking ${\mathcal B} = [-1,1]^4$ then yields an asymptotic formula for $|M^0_2(k)|$, and more generally one can obtain $|M^\lambda_2(k)|$ by adding and subtracting appropriate shifts of ${\mathcal B}$.  The asymptotic formula in \cite{K} is defined in terms of an unusual singular measure supported on the Zariski-closed subset of $\R^4$ corresponding to singular matrices.  The explicit computation of this measure is roughly analogous to our case-by-case considerations in Section~\ref{section: key proof}, modulo the significant complications of carrying error terms.  Our techniques are more elementary, but Katznelson's results apply in theory to matrices of any size, whereas our methods become unwieldy even for $3 \times 3$ matrices.

In the case of $n\times n$ matrices with entries independently chosen from a Gaussian distribution, a great deal more is known.  Edelman \cite{E} has computed the exact distributions of the real and complex eigenvalues for any $n$, as well as the number of real eigenvalues (for instance, the probability of having all $n$ eigenvalues real is precisely $2^{-n(n-1)/4}$).  As $n \to \infty$, the real eigenvalues, suitably rescaled by a factor of $1/\sqrt{n}$, converge to the uniform distribution on $[-1,1]$.  Similarly, the complex eigenvalues converge to the ``circular law'' predicted by Girko \cite{Girko}, namely the uniform distribution on the unit disk centered at the origin.  Very recently, Tao and Vu \cite{TV} have shown that the circular law is universal: one can replace the Gaussian distribution by an arbitrary distribution with mean 0 and variance 1. Similar results have been established for random symmetric matrices with entries independently chosen from a Gaussian distribution (the ``Wigner law'') or from other distributions.

Those who are interested in the connections between analytic number theory and random matrix theory might wonder whether those connections are related to the present paper. The matrices in that context, however, are selected from classical matrix groups, such as the group of $n\times n$ Hermitian matrices, randomly according to the Haar measures on the groups. The relationship to our results is therefore minimal.

\section{Preliminaries about matrices}

We begin with some elementary observations about $2\times 2$ matrices that will simplify our computations.  The first lemma explains why the functions $V$ and $W$ are supported only on $[-2,2]$.

\begin{lemma}
Any eigenvalue of a matrix in $\M_2(k)$ is bounded in absolute value by $2k$. Any eigenvalue of a matrix in $\Ntwo$ is bounded in absolute value by~2.
\label{lemma: baby Gershgorin}
\end{lemma}

\begin{proof}
We invoke Gershgorin's ``circle theorem'' \cite{G}, a standard result in spectral theory: let $M = (m_{ij})$ be an $n\times n$ matrix, and let $D(z,r)$ denote the disk of radius $r$ around the complex number $z$. Then Gershgorin's theorem says that all of the eigenvalues of $M$ must lie in the union of the disks
\[
D\bigg( m_{11}, \sum_{\substack{1\le j\le n \\ j\ne1}} |m_{1j}| \bigg), \quad D\bigg( m_{22}, \sum_{\substack{1\le j\le n \\ j\ne2}} |m_{2j}| \bigg), \quad \dots, \quad D\bigg( m_{nn}, \sum_{\substack{1\le j\le n \\ j\ne n}} |m_{nj}| \bigg).
\]
In particular, if all of the entries of $M$ are bounded in absolute value by $B$, then all the eigenvalues are bounded in absolute value by $nB$.
\end{proof}

The key to the precise enumeration of $\M_2^\lambda(k)$ is the simple structure of singular integer matrices:

\begin{lemma} 
For any singular matrix $M\in \M_2(k)$, either at least two entries of $M$ equal zero, or else there exist nonzero integers $a,b,c,d$ with $(a,b)=1$ such that
\begin{equation}
M = \begin{pmatrix} ac & bc \\ ad & bd \end{pmatrix}.
\label{represent}
\end{equation}
Moreover, this representation of $M$ is unique up to replacing each of $a,b,c,d$ by its negative.
\label{baby.enumeration}
\end{lemma}

\begin{proof}
If one of the entries of $M$ equals zero, then a second one must equal zero as well for the determinant to vanish. Otherwise, given
\[
M = \begin{pmatrix} m_{11} & m_{12} \\ m_{21} & m_{22} \end{pmatrix}
\]
with none of the $m_{ij}$ equal to zero, define $c=(m_{11},m_{12})$, and set $a=m_{11}/c$ and $b=m_{12}/c$, so that $(a,b)=1$. Since $M$ is singular, the second row of $m$ must be a multiple of the first row---that is, there exists a real number $d$ such that $m_{21}=ad$ and $m_{22}=bd$. Since $a$ and $b$ are relatively prime, moreover, $d$ must in fact be an integer.

This argument shows that every such matrix has one such representation. If
\[
M = \begin{pmatrix} a'c' & b'c' \\ a'd' & b'd' \end{pmatrix}
\]
is another such representation, then $(a',b')=1$ implies $(a'c',b'c')=|c'|$, which shows that $|c'|=c$; the equalities $|a'|=|a|$, $|b'|=|b|$, and $|d'|=|d|$ follow quickly.
\end{proof}

For a $2\times 2$ matrix $M=\big( \genfrac{}{}{0pt}{1}{a}{b} \genfrac{}{}{0pt}{1}{c}{d} \big),$ we define $\disc M = (\tr M)^2 - 4 \det M = (a-d)^2 + 4bc.$  It is easily seen that this is the discriminant of the characteristic polynomial of $M$.  We record the following elementary facts, which will be useful in the proof of Lemma~\ref {lemma: double eigenvalue} and Proposition~\ref {F_W proposition}.

\begin{lemma} Let $M$ be a $2 \times 2$ matrix with real entries.
\begin{enumerate}[(a)]
\item $M$ has repeated eigenvalues if and only if $\disc M = 0$.
\item $M$ has real eigenvalues if and only if $\disc M \ge 0$.
\item $\det M < 0$ if and only if $M$ has two real eigenvalues of opposite sign.
\item If $\det M > 0$ and $\disc M \ge 0$, then the eigenvalues of $M$ have the same sign as $\tr M$. 
\end{enumerate}
\label{discriminant}
\end{lemma}
\begin{proof}
Let $\lambda_1, \lambda_2$ denote the eigenvalues of $M$, so that $\tr M = \lambda_1 + \lambda_2$, $\det M = \lambda_1 \lambda_2$, and $\disc M = (\lambda_1 - \lambda_2)^2$, each of which is real.  Parts (a), (b) and (d) follow immediately from these observations, and part (c) from the fact that $\lambda_2 = \overline{\lambda_1}$ if $\lambda_1, \lambda_2$ are complex.
\end{proof}

The next lemma gives a bound for the probability of a matrix having repeated eigenvalues.  It is natural to expect this probability to converge to 0 as $k$ increases, and indeed such a result was obtained in \cite{HLM} for matrices of arbitrary size.  We give a simple proof of a stronger bound for the $2\times 2$ case, as well as an analogous qualitative statement for real matrices which will be helpful in the proof of Theorem~\ref{theorem: distribution of real eigenvalues}.

\begin{lemma}
The number of matrices in $\M_2(k)$ with a repeated eigenvalue is $\ll_\ep k^{2+\ep}$ for every $\ep>0$. The probability that a random matrix in $\Ntwo$ has a repeated eigenvalue or is singular is~$0$.
\label{lemma: double eigenvalue}
\end{lemma}
\begin{proof}
By Lemma~\ref{discriminant}(a), the $2\times2$ matrix
$\big( \genfrac{}{}{0pt}{1}{a}{c} \genfrac{}{}{0pt}{1}{b}{d} \big)$
has a double eigenvalue if and only if $4bc = -(a-d)^2$. For matrices in $\Ntwo$ this is easily seen to be a zero-probability event, as is the event that $\det M = ad-bc=0$.

For matrices in $\M_2(k)$, we enumerate how many can satisfy $4bc = -(a-d)^2$. If $a=d$ then there are $4k+1$ choices for $b,c\in\{-k,\dots,k\}$; otherwise there are at most $2\tau((a-d)^2/4)$ choices if $a\equiv d\mod2$ and no choices otherwise.  (Here $\tau(n)$ is the number-of-divisors function; the factor of 2 comes from the fact that $b$ and $c$ can be positive or negative, while the ``at most'' is due to the fact that not all factorizations of $-(a-d)^2$ result in two factors not exceeding $k$.)  Therefore the number of matrices in $\M_2(k)$ with a repeated eigenvalue is at most
\begin{equation}
\sum_{|a|\le k} (4k+1) + 2 \sum_{|a|\le k} \sum_{\substack{|d|\le k  \\ d\neq a\\ a\equiv d\mod2}} \tau\bigg( \frac{(a-d)^2}4 \bigg) \ll_\ep k^{2+\ep},
\label{could use Perron}
\end{equation}
where the inequality follows from $(a-d)^2/4 \le k^2$ and the well-known fact that $\tau(n) \ll_\ep n^\ep$ for any $\ep>0$ (see for instance \cite[p.~56]{Magic}).
\end{proof}

%We remark that a more careful treatment of equation~\eqref{could use Perron} using Perron's formula would result in the following statement: the number of matrices in $\M_2(k)$ with a repeated eigenvalue is
%\[
%k^2(a_2\log^2k + a_1\log k+a_0) + o(k^{3/2})
%\]
%for certain explicit constants $a_j$ (for example, $a_2 = 12/\pi^2$, although $a_0$ and $a_1$ are more complicated).  However, the easy upper bound in the lemma is enough for our purposes.
%Note that Theorem~\ref{theorem: particular integer eigenvalue} provides an exact asymptotic for the number of singular matrices in $\M_2(k)$ in the case $\lambda=0$.

\section{Enumeration theorems for integer eigenvalues}
\label{enumeration}
Let $\mu(n)$ be the M\"obius function, characterized by the identity
\begin{equation}
\sum_{d\mid n} \mu(d) = \begin{cases}1, &\text{if }n=1, \\ 0, &\text{if }n>1.\end{cases}
\label{mobius}
\end{equation}
The well-known Dirichlet series identity $1/{\zeta(s)} = \sum_{n=1}^\infty \mu(n) 
n^{-s}$ is valid for $\Re{s} > 1$ (see, for example, \cite[Corollary 1.10]{Magic}).  In particular, $\sum \mu(n)/n^2 = 6/\pi^2$, and we can estimate the tail of this series (using $|\mu(n)| \le 1$) to obtain the quantitative estimate
\begin{equation}
\sum_{d\le k} \frac{\mu(d)}{d^2} = \frac{6}{\pi^2} + O\!\left(\frac{1}{k}\right).
\label{mobius sum}
\end{equation}

\begin{lemma}
For nonzero integers $a$, $b$ and parameters $k$, $\lambda$, define the function
\begin{equation}
N_{k,\lambda}(a,b) = \#\{(c,d) \in \Z^2,\, c\ne0,\, d\ne0: |ac+\lambda|,  |bc|,|ad|, |bd+\lambda| \le k \}.
\label{definition of N lambda}
\end{equation}
Then
\[
\big| \M^\lambda_2(k) \big| = 4 \sum_{d\le k} \mu(d) \sum_{1\le \alpha< \beta \le k/d} N_{k,\lambda}(d\alpha,d\beta) + O(k^2),
\]
where the implied constant is independent of $\lambda$ and $k$.
\label{initial count and mu insertion}
\end{lemma}

\begin{proof}
Fix an integer $0 \le \lambda \le 2k$, and let $M \in \M^\lambda_2(k)$, so that $M-\lambda I$ is singular. By Lemma~\ref{baby.enumeration}, either at least two entries of $M-\lambda I$ equal zero, or else $M-\lambda I$ has exactly two representations of the form \eqref{represent}. In the former case, there are $2k+1$ choices for each of the two potentially nonzero entries, hence $O(k^2)$ such matrices in total (even taking into account the several different choices of which two entries are nonzero). In the latter case, there are exactly two corresponding quadruples $a,b,c,d$ of integers as in Lemma~\ref{baby.enumeration}. Taking into account that each entry of $M$ must be at most $k$ in absolute value, we deduce that
\begin{align*}
\big| \M^\lambda_2(k) \big| &=  \tfrac12\, \#\big\{(a,b,c,d) \in \Z^4: a,b,c,d\ne0, \notag \\
&\qquad (a,b)=1,\, |ac+\lambda|,  |bc|,|ad|, |bd+\lambda| \le k\big\} + O(k^2) \notag \\
&= \tfrac12 \sum_{\substack{1\le |a|,|b| \le k \\ (a,b)=1}} N_{k,\lambda}(a,b) + O(k^2),
\label{shifted.count}
\end{align*}
where $N_{k,\lambda}(a,b)$ is defined as above.

Because of the symmetries $N_{k,\lambda}(a,b)=N_{k,\lambda}(-a,b)=N_{k,\lambda}(a,-b)=N_{k,\lambda}(-a,-b)$, we have
\[
\big| \M^\lambda_2(k) \big| = 2 \sum_{\substack{1\le a,b \le k \\ (a,b)=1}} N_{k,\lambda}(a,b) + O(k^2).
\]
The only term in the sum where $a=b$ is the term $a=b=1$, and for all other terms we can invoke the additional symmetry $N_{k,\lambda}(a,b)=N_{k,\lambda}(b,a)$, valid by switching the roles of $c$ and $d$ in the definition~\eqref{definition of N lambda} of $N_{k,\lambda}(a,b)$. We obtain
\begin{align*}
\big| \M^\lambda_2(k) \big| &=  4 \sum_{\substack{1\le a<b \le k \\ (a,b)=1}} N_{k,\lambda}(a,b) + 2N_{k,\lambda}(1,1) + O(k^2) \\
&=  4 \sum_{\substack{1\le a < b \le k \\ (a,b)=1}} N_{k,\lambda}(a,b) + O(k^2),
\end{align*}
where the last step used the fact that $N_{k,\lambda}(1,1) \le \#\{(c,d)\in\Z^2 : |c|,|d|\le k\} \ll k^2$.

Using the characteristic property of the M\"obius function \eqref{mobius}, we can write the last expression as
\begin{align*}
\big| \M^\lambda_2(k) \big| &= 4 \sum_{1\le a< b \le k} N_{k,\lambda}(a,b) \sum_{d\mid (a,b)} \mu(d) + O(k^2) \\
 &= 4 \sum_{d\le k} \mu(d) \sum_{\substack{1\le a< b \le k \\ d\mid a,\, d\mid b}} N_{k,\lambda}(a,b) + O(k^2) \\
&= 4 \sum_{d\le k} \mu(d) \sum_{1\le \alpha< \beta \le k/d} N_{k,\lambda}(d\alpha,d\beta) + O(k^2),
\end{align*}
as claimed.
\end{proof}

\begin{lemma}
Let $k$ and $\lambda$ be integers with $0\le \lambda\le 2k$, and let $x$ and $y$ be integers with $1\le x\le y\le k$. Then
\[
N_{k,\lambda}(x,y) = k^2 C\big( \tfrac \lambda k; x,y \big) D\big( \tfrac \lambda k; x,y \big) + O\big(\tfrac{k}{y}\big),
\]
where
\begin{align}
C(\delta;x,y) &= \max \big\{ 0, \min \big\{ \tfrac{1-\delta}x + \tfrac1y , \tfrac2y \big\} \big\}, \notag \\
D(\delta;x,y) &= \min\big\{ \tfrac{1-\delta}y + \tfrac1x, \tfrac2y \big\}.
\label{CD definition}
\end{align}
\label{converting from N to CD lemma}
\end{lemma}

\begin{proof}
We have
\begin{align*}
N_{k,\lambda}(x,y) &= \#\{(c,d) \in \Z^2,\, c\ne0,\, d\ne0: |xc+\lambda|,  |yc|,|xd|, |yd+\lambda| \le k \} \\
&= \#\{c \in \Z,\, c\ne0: -k \le xc+\lambda\le k,\, -k \le yc\le k \} \\
&\qquad{}\times \#\{d \in \Z,\, d\ne0: -k \le xd\le k,\, -k \le yd+\lambda\le k \}.
\end{align*}
Since $x$ and $y$ are positive, we can rewrite this product as
\begin{align}
N_{k,\lambda}(x,y) &= \#\{c \in \Z,\, c\ne0: (-k-\lambda)/x \le c\le (k-\lambda)/x,\, -k/y \le c\le k/y \} \notag \\
&\qquad {}\times \#\{d \in \Z,\, d\ne0: -k/x \le d\le k/x,\, (-k-\lambda)/y \le d\le (k-\lambda)/y \} \notag \\
&= \#\big\{c \in \Z,\, c\ne0: -k/y \le c \le \min\{ (k-\lambda)/x, k/y \} \big\} \label{eight inequalities} \\
&\qquad {}\times \#\big\{d \in \Z,\, d\ne0: \max\{ -k/x, (-k-\lambda)/y \} \le d\le (k-\lambda)/y \big\}, \notag 
\end{align}
where we have used $\lambda\ge0$ and $x\le y$ to slightly simplify the inequalities. The first factor on the right-hand side of equation~\eqref{eight inequalities} is
\[
\min\{ (k-\lambda)/x, k/y \} - (-k/y) + O(1) = k \min \big\{ (1-\tfrac\lambda k)/x + 1/y, 2/y \big\} + O(1)
\]
if this expression is positive, and 0 otherwise; it is thus precisely $k C(\tfrac\lambda k;x,y) + O(1)$. Similarly, the second factor on the right-hand side of equation~\eqref{eight inequalities} is
\[
(k-\lambda)/y - \max\{ -k/x, (-k-\lambda)/y \} + O(1) = k \min \{ (1-\tfrac\lambda k)/y + 1/x, 2/y \big\} + O(1)
\]
(note that this expression is always positive under the hypotheses of the lemma), which is simply $k D(\tfrac\lambda k;x,y) + O(1)$.  Multiplying these two factors yields
\[
N_{k,\lambda}(x,y) =  k^2 C\big( \tfrac \lambda k; x,y \big) D\big( \tfrac \lambda k; x,y \big) + k \cdot O( C\big( \tfrac \lambda k; x,y \big ) + D\big( \tfrac \lambda k; x,y \big)) + O(1).
\]
The lemma follows upon noting that both $ C( \tfrac \lambda k; x,y )$ and  $D( \tfrac \lambda k; x,y )$ are $\ll 1/y$ by definition, so that the second summand becomes simply $O(\tfrac k y)$, and the $O(1)$ term may be subsumed into $O(\tfrac k y)$ since $y \le k$.
\end{proof}

We have already used the trivial estimate
\[
\sum_{L\le\alpha< U} 1 = (U-L) + O(1),
\]
provided $0 < L < U$.  We will also use, without further comment, the estimates
\begin{equation*}
\sum_{L\le\alpha< U} \tfrac1\alpha = \log \tfrac UL + O\big( \tfrac1L \big);
\quad\text{in particular,}\quad
\sum_{1\le\alpha< U} \tfrac1\alpha = \log U + O(1);
\label{reciprocal sum}
\end{equation*}
and
\begin{equation*}
\sum_{L\le\alpha< U} \tfrac1{\alpha^2} = \tfrac1L - \tfrac1U + O\big( \tfrac1{L^2} \big);
\quad\text{in particular,}\quad
\sum_{L\le\alpha< U} \tfrac1{\alpha^2} = O\big( \tfrac1L \big).
\label{square reciprocal sum}
\end{equation*}
These estimates (also valid for $0 < L < U$) follow readily from comparison to the integrals $\int_L^U 1/x \,dx$ and
$\int_L^U 1/x^2 \, dx$.

Most of the technical work in proving Theorem~\ref{theorem: particular integer eigenvalue} lies in establishing an estimate on a sum of the form $\sum_{1\le\alpha < \beta} C(\delta;\alpha,\beta) D(\delta;\alpha,\beta)$ for a fixed $\beta$. The following proposition provides an asymptotic formula for this sum; we defer the proof until the next section. Assuming this proposition, though, we can complete the proof of Theorem~\ref{theorem: particular integer eigenvalue}, as well as Corollary~\ref{corollary: all integer eigenvalues}.

\begin{proposition}
Let $\beta\ge1$ and $0\le\delta\le 2$ be real numbers, and let $C$ and $D$ be the functions defined in equation~\eqref{CD definition}. Then
\[
\sum_{1\le\alpha < \beta} C(\delta;\alpha,\beta) D(\delta;\alpha,\beta) = \frac{V(\delta)}\beta + O\bigg( \frac{1+\log\beta}{\beta^2} \bigg),
\]
where $V(\delta)$ was defined in equation~\eqref{definition of V(delta)}.
\label{key proposition}
\end{proposition}

\begin{proof}[Proof of Theorem~\ref{theorem: particular integer eigenvalue} assuming Proposition~\ref{key proposition}]
The functions $C$ and $D$ defined in equation~\eqref{CD definition} are homogeneous of degree $-1$ in the variables $x$ and $y$, so that Lemma~\ref {converting from N to CD lemma} implies
\[
N_{k,\lambda}(d\alpha,d\beta) = \frac{k^2}{d^2} C\big( \tfrac \lambda k; \alpha,\beta \big) D\big( \tfrac \lambda k; \alpha,\beta \big) + O\big(\tfrac{k}{d\beta}).
\]
Inserting this formula into the conclusion of Lemma~\ref{initial count and mu insertion} yields
\[
\big| \M^\lambda_2(k) \big| = 4k^2 \sum_{d\le k} \frac{\mu(d)}{d^2} \sum_{1\le \alpha< \beta \le k/d} C\big( \tfrac \lambda k; \alpha,\beta \big) D\big( \tfrac \lambda k; \alpha,\beta \big) + O\bigg(\sum_{d\le k} \sum_{1 \le \alpha < \beta \le k/d} \frac{k}{d\beta}\bigg) + O(k^2).
\]
We bound the first error term by summing over $1 \le \alpha < \beta$ to obtain
\[
\sum_{d\le k} \sum_{1 \le \alpha < \beta \le k/d} \frac{k}{d\beta} \le  \sum_{d\le k} \sum_{1 < \beta \le k/d} \frac{k}{d} < \sum_{d\le k} \frac{k^2}{d^2} \ll k^2,
\]
so that we have the estimate
\begin{equation}
\big| \M^\lambda_2(k) \big| = 4k^2 \sum_{d\le k} \frac{\mu(d)}{d^2} \sum_{1\le \alpha< \beta \le k/d} C\big( \tfrac \lambda k; \alpha,\beta \big) D\big( \tfrac \lambda k; \alpha,\beta \big) + O(k^2).
\label{on hold for key proposition}
\end{equation}
We now apply Proposition~\ref{key proposition} to obtain
\begin{align*}
\big| \M^\lambda_2(k) \big| &= 4k^2 \sum_{d\le k} \frac{\mu(d)}{d^2} \sum_{1\le \beta \le k/d} \bigg( \frac{V(\lambda/k)}\beta + O\bigg( \frac{1+\log\beta}{\beta^2} \bigg) \bigg) + O(k^2) \\
&= 4k^2 \sum_{d\le k} \frac{\mu(d)}{d^2} \bigg( V(\lambda/k) \bigg( \log \frac k d + O(1) \bigg) + O(1) \bigg) + O(k^2) \\
&= 4k^2 \bigg( V(\lambda/k)\log k \sum_{d\le k} \frac{\mu(d)}{d^2} + O\bigg( \sum_{d\le k} \frac{\log d}{d^2} \bigg) \bigg) + O(k^2) \\
&= 4k^2 \bigg( V(\lambda/k)\log k \bigg( \frac6{\pi^2} + O\bigg( \frac1k \bigg) \bigg) + O(1) \bigg) + O(k^2) \\
&= \frac{24V(\lambda/k)}{\pi^2}k^2\log k + O(k^2),
\end{align*}
where we have used equation~\eqref{mobius sum} and the fact that $\sum 1/n^2$ and $\sum(\log n)/n^2$ are convergent (so the partial sums are uniformly bounded).
\end{proof}

\begin{proof}[Proof of Corollary~\ref{corollary: all integer eigenvalues} from Theorem~\ref{theorem: particular integer eigenvalue}]
Note that for any $M\in \M_2(k)$, if one eigenvalue is an integer then they both are (since the trace of $M$ is an integer). Thus if we add up the cardinalities of all of the $\M_2^\lambda(k)$, we get twice the cardinality of $\M_2^\Z(k)$, except that matrices with repeated eigenvalues only get counted once. However, the number of such matrices is $\ll_\ep k^{2+\ep}$ by Lemma~\ref {lemma: double eigenvalue}. Therefore
\begin{align*}
2|\M_2^\Z(k)| &= \sum_{\lambda\in\Z} |\M_2^\lambda(k)| + O_\ep(k^{2+\ep}) \\
&= \sum_{-2k\le\lambda\le2k} \bigg( \frac{24V(\lambda/k)}{\pi^2} k^2\log k + O(k^2) \bigg) + O_\ep(k^{2+\ep}) \\
&= \frac{24k^3\log k}{\pi^2} \sum_ {-2k\le\lambda\le2k} \frac{V(\lambda/k)}k + O(k^3).
\end{align*}
The sum is a Riemann sum of a function of bounded variation, so this becomes
\[
2|\M_2^\Z(k)| = \frac{24k^3\log k}{\pi^2} \bigg( \int_{-2}^2 V(\delta)\,d\delta + O\big(\tfrac1k\big) \bigg) + O(k^3).
\]
The corollary then follows from the straightforward computation of the integral $\int_{-2}^2 V(\delta)\,d\delta = \tfrac{4}{9}(7\sqrt{2} + 4 + 3\log(\sqrt{2}+1))$, noting that $\log(\sqrt{2}-1) = -\log(\sqrt{2}+1)$.
\end{proof}

\section{Proving the key proposition}
\label{section: key proof}
It remains to prove Proposition~\ref{key proposition}.  Recalling that the functions $C$ and $D$ defined in equation~\eqref{CD definition} are formed by combinations of minima and maxima, we need to separate our arguments into several cases depending on the range of $\delta$.  The following lemma addresses a sum that occurs in two of these cases ($0 < \delta < 1$ and $1 < \delta < \sqrt{2}$).  Note that because of the presence of terms like $\log(\delta-1)$ in the formula for $V(\delta)$, we need to exercise some caution near $\delta=1$.

\begin{lemma}
Let $\beta\ge1$ and $0\le\delta \le \sqrt{2}$ be real numbers, with $\delta\ne1$. Then
\begin{multline*}
\sum_{\max\{1,|1-\delta|\beta\}\le\alpha<(1+\delta)^{-1}\beta} \bigg( \frac{1-\delta}\alpha + \frac1\beta \bigg) \frac2\beta \\
{}= \frac2\beta \bigg( \frac1{1+\delta} - |1-\delta| - (1-\delta)\log|1-\delta^2| \bigg) + O\bigg( \frac{1+\log\beta}{\beta^2} \bigg).
\end{multline*}
\label{lemma where lower summation limit could be 1}
\end{lemma}

\begin{proof}
Suppose first that $|1-\delta|\beta\ge1$. Then the sum in question is
\begin{align*}
\frac{2(1-\delta)}\beta \sum_{|1-\delta|\beta\le\alpha<(1+\delta)^{-1}\beta} & \frac1\alpha + \frac2{\beta^2} \sum_{|1-\delta|\beta\le\alpha<(1+\delta)^{-1}\beta} 1 \\
&{}= \frac{2(1-\delta)}\beta \bigg( \log\frac{\beta/(1+\delta)}{|1-\delta|\beta} + O\bigg( \frac1{|1-\delta|\beta} \bigg) \bigg) \\
&\qquad{}+ \frac2{\beta^2} \bigg( \frac\beta{1+\delta} - |1-\delta|\beta + O(1) \bigg) \\
&{}= \frac2\beta \bigg( \frac1{1+\delta} - |1-\delta| - (1-\delta)\log|1-\delta^2| \bigg) + O\bigg( \frac1{\beta^2} \bigg),
\end{align*}
which establishes the lemma in this case. On the other hand, if $|1-\delta|\beta<1$ then the sum in question is
\begin{align*}
\frac{2(1-\delta)}\beta & \sum_{1\le\alpha<(1+\delta)^{-1}\beta} \frac1\alpha + \frac2{\beta^2} \sum_{1\le\alpha<(1+\delta)^{-1}\beta} 1 \\
&{}= \frac{2(1-\delta)}\beta \bigg( \log\frac\beta{1+\delta} + O(1) \bigg) + \frac2{\beta^2} \bigg( \frac\beta{1+\delta} + O(1) \bigg) \\
&{}= \frac2\beta \bigg( \frac1{1+\delta} - (1-\delta)\log(1+\delta) \bigg) + O\bigg( \frac1{\beta^2} + \frac{|1-\delta|\log\beta}\beta \bigg).
\end{align*}
We subtract $2(|1-\delta| + (1-\delta)\log|1-\delta|)/\beta$ from the main term and compensate in the error term to obtain
\begin{align*}
\frac{2(1-\delta)}\beta & \sum_{1\le\alpha<(1+\delta)^{-1}\beta} \frac1\alpha + \frac2{\beta^2} \sum_{1\le\alpha<(1+\delta)^{-1}\beta} 1 \\
&{}= \frac2\beta \bigg( \frac1{1+\delta} - |1-\delta| - (1-\delta)\log|1-\delta^2| \bigg) \\
&\qquad{} + O\bigg( \frac1{\beta^2} + \frac{|1-\delta|\log\beta}\beta + \frac{|1-\delta|+|1-\delta|\log|1-\delta|^{-1})}\beta \bigg) \\
&{}= \frac2\beta \bigg( \frac1{1+\delta} - |1-\delta| - (1-\delta)\log|1-\delta^2| \bigg) + O\bigg( \frac{1+\log\beta}{\beta^2} + \frac{|1-\delta|\log|1-\delta|^{-1})}\beta \bigg),
\end{align*}
since we are working with the assumption that $|1-\delta|< 1/\beta$.
Because the function $t\log t^{-1}$ is increasing on the interval $(0,1/e)$ and bounded on the interval $(0,1]$, we have $|1-\delta|\log|1-\delta|^{-1} < (1/\beta)\log\beta$ if $\beta>e$ and $|1-\delta|\log|1-\delta|^{-1} \ll1 \ll 1/\beta$ if $1\le \beta\le e$. In either case, the last error term can be simplified to $O((1+\log\beta)/\beta^2)$, which establishes the lemma in second case.
\end{proof}

\begin{proof}[Proof of Proposition~\ref{key proposition}]
We consider separately the four cases corresponding to the different parts of the definition~\eqref{definition of V(delta)} of $V(\delta)$.

{\em$\bullet$ Case 1: $0\le\delta<1$.} In this case we have $0 < 1-\delta < (1+\delta)^{-1} \le 1$ and
\[
C(\delta;\alpha,\beta) = \begin{cases}
\frac2\beta, & \text{if } \alpha\le(1-\delta)\beta, \\
\frac{1-\delta}\alpha + \frac1\beta, & \text{if } (1-\delta)\beta\le\alpha
\end{cases}
\quad\text{and}\quad
D(\delta;\alpha,\beta) = \begin{cases}
\frac2\beta, & \text{if } \alpha\le(1+\delta)^{-1}\beta, \\
\frac{1-\delta}\beta + \frac1\alpha, & \text{if } (1+\delta)^{-1}\beta\le\alpha.
\end{cases}
\]
Therefore
\begin{multline}
\sum_{1\le\alpha<\beta} C(\delta;\alpha,\beta) D(\delta;\alpha,\beta) = \sum_{1\le\alpha<(1-\delta)\beta} \frac2\beta \cdot \frac2\beta+ \sum_{\max\{1, (1-\delta)\beta\} \le \alpha < (1+\delta)^{-1}\beta} \bigg( \frac{1-\delta}\alpha + \frac1\beta \bigg) \frac2\beta \\
\qquad{} + \sum_{(1+\delta)^{-1}\beta \le \alpha < \beta} \bigg( \frac{1-\delta}\alpha + \frac1\beta \bigg) \bigg( \frac{1-\delta}\beta + \frac1\alpha \bigg).
\label{case 1 first step}
\end{multline}
(The first sum might be empty, but this does not invalidate the argument that follows.) The first sum is simply
\[
\frac4{\beta^2} \sum_{1\le\alpha<(1-\delta)\beta} 1 = \frac4{\beta^2} \big( (1-\delta)\beta + O(1) \big) = \frac{4(1-\delta)}\beta + O\bigg( \frac1{\beta^2} \bigg).
\]
By Lemma~\ref{lemma where lower summation limit could be 1}, the second sum is
\begin{multline*}
\sum_{\max\{1,(1-\delta)\beta\}\le\alpha<(1+\delta)^{-1}\beta} \bigg( \frac{1-\delta}\alpha + \frac1\beta \bigg) \frac2\beta \\
{}= \frac2\beta \bigg( \frac1{1+\delta} + \delta-1 - (1-\delta)\log (1-\delta^2) \bigg) + O\bigg( \frac{1+\log\beta}{\beta^2} \bigg),
\end{multline*}
while the third sum is
\begin{align}
\sum_{(1+\delta)^{-1}\beta \le \alpha < \beta} & \bigg( \frac{1-\delta}\alpha + \frac1\beta \bigg) \bigg( \frac{1-\delta}\beta + \frac1\alpha \bigg) \notag \\
&= (1-\delta) \sum_{(1+\delta)^{-1}\beta \le \alpha < \beta} \frac1{\alpha^2} + \frac{\delta^2-2\delta+2}\beta \sum_{(1+\delta)^{-1}\beta \le \alpha < \beta} \frac1\alpha + \frac{1-\delta}{\beta^2} \sum_{(1+\delta)^{-1}\beta \le \alpha < \beta} 1 \notag \\
&= (1-\delta) \bigg( \frac{1+\delta}\beta - \frac1\beta + O\bigg( \frac1{\beta^2} \bigg) \bigg) + \frac{\delta^2-2\delta+2}\beta \bigg( \log \frac{\beta}{(1+\delta)^{-1}\beta} + O\bigg( \frac1\beta \bigg) \bigg) \notag \\
&\qquad{}+ \frac{1-\delta}{\beta^2} \bigg( \beta - \frac{\beta}{1+\delta} + O(1) \bigg) \notag \\
&= \frac1\beta \bigg( 2-\delta^2 - \frac2{1+\delta} + (\delta^2-2\delta+2)\log(1+\delta) \bigg) + O\bigg( \frac1{\beta^2} \bigg).
\label{double duty evaluation}
\end{align}
This case of the proposition then follows from equation~\eqref{case 1 first step}, on noting that
\begin{multline*}
4-4\delta + \frac2{1+\delta} + 2\delta-2 - 2(1-\delta)\log(1-\delta^2) + 2-\delta^2 - \frac2{1+\delta} + (\delta^2-2\delta+2)\log(1+\delta) \\
= 4 - 2\delta - \delta^2 + \delta^2 \log(1+\delta) - 2(1-\delta)\log (1-\delta).
\end{multline*}

{\em$\bullet$ Case 2: $\delta=1$.} In this case we have
\[
C(\delta;\alpha,\beta) = \frac1\beta
\quad\text{and}\quad
D(\delta;\alpha,\beta) = \begin{cases}
2/\beta, & \text{if } \alpha\le\beta/2, \\
1/\alpha, & \text{if } \beta/2\le\alpha.
\end{cases}
\]
Therefore
\begin{align*}
\sum_{1\le\alpha<\beta} C(\delta;\alpha,\beta) D(\delta;\alpha,\beta) &= \sum_{1 \le \alpha < \beta/2} \frac1\beta \cdot \frac2\beta + \sum_{\beta/2 \le \alpha < \beta} \frac1\beta \cdot \frac1\alpha \\
&{}= \frac2{\beta^2} \bigg( \frac\beta2 + O(1) \bigg) + \frac1\beta \bigg( \log\frac\beta{\beta/2} + O\bigg( \frac1\beta \bigg) \bigg) \\
&= \frac{1 + \log2}\beta + O\bigg( \frac1{\beta^2} \bigg),
\end{align*}
as desired.

{\em$\bullet$ Case 3: $1<\delta\le\sqrt2$.} In this case we have
\[
C(\delta;\alpha,\beta) = \begin{cases}
0, & \text{if } \alpha\le(\delta-1)\beta, \\
\frac{1-\delta}\alpha + \frac1\beta, & \text{if } (\delta-1)\beta\le\alpha
\end{cases}
\quad\text{and}\quad
D(\delta;\alpha,\beta) = \begin{cases}
\frac2\beta, & \text{if } \alpha\le(\delta+1)^{-1}\beta, \\
\frac{1-\delta}\beta + \frac1\alpha, & \text{if } (\delta+1)^{-1}\beta\le\alpha.
\end{cases}
\]
Therefore
\begin{multline}
\sum_{1\le\alpha<\beta} C(\delta;\alpha,\beta) D(\delta;\alpha,\beta) = \sum_{\max\{1, (\delta-1)\beta\} \le \alpha < (\delta+1)^{-1}\beta} \bigg( \frac{1-\delta}\alpha + \frac1\beta \bigg) \frac2\beta \\
\qquad{}+ \sum_{(\delta+1)^{-1}\beta \le \alpha < \beta} \bigg( \frac{1-\delta}\alpha + \frac1\beta \bigg) \bigg( \frac{1-\delta}\beta + \frac1\alpha \bigg).
\label{case 3 first step}
\end{multline}
(We note that $(\delta-1)\beta \le (\delta+1)^{-1}\beta$ for $\delta$ between 1 and $\sqrt2$. For very small $\beta$ we might have $1>(\delta+1)^{-1}\beta$, in which case the first sum is empty, but that does not invalidate the argument that follows.) By Lemma~\ref{lemma where lower summation limit could be 1}, the first sum is
\begin{multline*}
\sum_{\max\{1,(\delta-1)\beta\}\le\alpha<(1+\delta)^{-1}\beta} \bigg( \frac{1-\delta}\alpha + \frac1\beta \bigg) \frac2\beta \\
{}= \frac2\beta \bigg( \frac1{1+\delta} + 1-\delta - (1-\delta)\log(\delta^2-1) \bigg) + O\bigg( \frac{1+\log\beta}{\beta^2} \bigg),
\end{multline*}
while the second sum has already been evaluated in equation~\eqref{double duty evaluation} above. This case of the proposition then follows from equation~\eqref{case 3 first step}, on noting that
\begin{multline*}
\frac2{1+\delta} + 2-2\delta - 2(1-\delta)\log(\delta^2-1) + 2-\delta^2 - \frac2{\delta+1} + (\delta^2-2\delta+2)\log(\delta+1) \\
= 4 - 2\delta - \delta^2 + \delta^2 \log(\delta+1) + 2(\delta-1)\log(\delta-1).
\end{multline*}

{\em$\bullet$  Case 4: $\sqrt2<\delta\le2$.} Just as in Case 3, we have
\[
C(\delta;\alpha,\beta) = \begin{cases}
0, & \text{if } \alpha\le(\delta-1)\beta, \\
\frac{1-\delta}\alpha + \frac1\beta, & \text{if } (\delta-1)\beta\le\alpha
\end{cases}
\quad\text{and}\quad
D(\delta;\alpha,\beta) = \begin{cases}
\frac2\beta, & \text{if } \alpha\le(\delta+1)^{-1}\beta, \\
\frac{1-\delta}\beta + \frac1\alpha, & \text{if } (\delta+1)^{-1}\beta\le\alpha.
\end{cases}
\]
However, the inequality $(\delta-1)\beta\le\alpha$ automatically implies that $(\delta+1)^{-1}\beta\le\alpha$ when $\delta\ge\sqrt2$. Therefore
\begin{equation*}
\sum_{1\le\alpha<\beta} C(\delta;\alpha,\beta) D(\delta;\alpha,\beta) = \sum_{(\delta-1)\beta \le \alpha < \beta} \bigg( \frac{1-\delta}\alpha + \frac1\beta \bigg) \bigg( \frac{1-\delta}\beta + \frac1\alpha \bigg).
\end{equation*}
(In this case we will not need to use the more precise lower bound $\max\{1,(\delta-1)\beta\} \le \alpha$.) This yields
\begin{align*}
\sum_{1\le\alpha<\beta} & C(\delta;\alpha,\beta) D(\delta;\alpha,\beta) \\
&= (1-\delta) \sum_ {(\delta-1)\beta \le \alpha < \beta} \frac1{\alpha^2} + \frac{\delta^2-2\delta+2}\beta \sum_ {(\delta-1)\beta \le \alpha < \beta} \frac1\alpha + \frac{1-\delta}{\beta^2} \sum_ {(\delta-1)\beta \le \alpha < \beta} 1 \\
&= (1-\delta) \bigg( \frac1{(\delta-1)\beta} - \frac1\beta + O\bigg( \frac1{(\delta-1)^2\beta^2} \bigg) \bigg) \\
&\qquad{}+ \frac{\delta^2-2\delta+2}\beta \bigg( \log \frac{\beta}{(\delta-1)\beta} + O\bigg( \frac1{(\delta-1)\beta} \bigg) \bigg) + \frac{1-\delta}{\beta^2} \bigg( \beta - (\delta-1)\beta + O(1) \bigg) \\
&= \frac1\beta \bigg( \delta^2 - 2\delta - (\delta^2-2\delta+2)\log(\delta-1) \bigg) + O\bigg( \frac1{\beta^2} \bigg),
\end{align*}
where the error terms have been simplified since $\delta-1$ is bounded away from~0.

This ends the proof of Proposition~\ref{key proposition}.
\end{proof}

\section{Distribution of real eigenvalues}
In proving Theorem~\ref{theorem: distribution of real eigenvalues}, it will be convenient to define the odd function
\begin{equation}
G(z) = \int_0^z -\log|t| \, dt = z(1-\log|z|), \label{G definition}
\end{equation}
whose relevance is demonstrated by the following lemma.

\begin{lemma}
If $B$ and $C$ are independent random variables uniformly distributed on $[-1,1]$, then the product $BC$ is a random variable whose distribution function is $F_{BC}(z) = \P(BC<z) = \tfrac{1}{2}(1+G(z))$ for $z \in [-1,1]$.
\label{product distribution}
\end{lemma}
Of course, for $z < -1$ we have $F_{BC}(z) = 0$, and likewise $F_{BC}(z) = 1$ for $z > 1$.
\begin{proof}
Note that $|B|$ and $|C|$ are uniformly distributed on $[0,1]$. For $0 \le z \le 1$, we easily check that $\P(|BC|<z) = \int_0^1 \min\{1,z/s\} \, ds = G(z)$.  Thus $|BC|$ is distributed on $[0,1]$ with density $f_{|BC|}(z) = -\log z$, and by symmetry $BC$ has density $f_{BC}(z) = -\tfrac{1}{2} \log|z|$ on $[-1,1]$.  The lemma follows upon computing $F_{BC}(z) = \int_{-1}^z f_{BC}(s) \, ds$.
\end{proof}

It will also be helpful to define the following functions, which are symmetric in $x$ and $y$:
\begin{align}
\nu_1(x,y) &= 1/2+G(xy)/2 + G((x-y)^2/4), \label{nu1 definition} \\
\nu_2(x,y) &= 1/2-G(xy)/2, \label{nu2 definition} \\
\nu(x,y)  &= \begin{cases}
\nu_1(x,y), &\text{if~}xy<1 \text{ and } x+y<0, \\
\nu_2(x,y), &\text{if~} xy<1 \text{ and } x+y>0, \\
1 + G((x-y)^2/4), &\text{if~} xy>1 \text{ and } x+y<0, \\
0, & \text{otherwise.}
\end{cases}
\label{nu definition}
\end{align}

To prove Theorem~\ref{theorem: distribution of real eigenvalues}, we first consider the distribution function $F_W(\delta) = \int_{t<\delta} W(t) \,dt$ associated to the density $W(\delta)$.  For a random matrix $M$ in $\Ntwo$ and a real number $\delta$, we will derive an expression for the expected number of real eigenvalues of $M$ falling below $\delta$, then differentiate it to obtain $W(\delta)$.

Since the set $\Ntwo$ is closed under negation, it is clear that $W(-\delta) = W(\delta)$, so it suffices to compute $W(\delta)$ for $\delta \in [0,2]$.  It turns out that our calculations for $F_W$ will be somewhat simplified by considering $F_W(-\delta)$ rather than $F_W(\delta)$.

\begin{proposition}
We have
\[
F_W(-\delta) = \frac{1}{4}\int_{-1+\delta}^{1+\delta} \int_{-1+\delta}^{1+\delta} \nu(x,y) \,dx\,dy
\]
for all $0\le \delta\le 2$, where $\nu$ is defined in equation~\eqref{nu definition}.
\label{F_W proposition}
\end{proposition}

\begin{proof}
Let us denote the entries of $M$ by the random variables $A$, $B$, $C$, $D$, which by assumption are independent and uniformly distributed in $[-1,1]$.  Let $\delta$ be fixed in the range $[0,2]$, and consider the shifted matrix $M' = M+\delta I$, which we write as
\[
M' = \begin{pmatrix} X & B \\ C & Y \end{pmatrix},
\]
where $B$, $C$ are as before, and $X$, $Y$ range independently and uniformly in $[-1+\delta,1+\delta]$.  Clearly the eigenvalues of $M$ less than $-\delta$ correspond to the negative (real) eigenvalues of $M'$.  By Lemma~\ref{lemma: double eigenvalue}, we are free to exclude the null set where $M'$ is singular or has repeated eigenvalues.  Outside of this null set, $M'$ has exactly one negative eigenvalue if and only if $\det M' = XY-BC < 0$, by Lemma~\ref{discriminant}(c). Likewise by Lemma~\ref{discriminant}(d), $M'$ has exactly two negative eigenvalues if and only if $XY-BC>0$ and $X+Y < 0$ and $\disc M' = (X-Y)^2 + 4BC > 0$. We thus have:
\begin{equation*}
F_W(-\delta) = \P(BC > XY) + 2\P(X+Y<0\text{ and } {-(X-Y)^2/4} < BC < XY).
\end{equation*}
We may express this probability as the average value
\[
F_W(-\delta) = \frac{1}{4}\int_{-1+\delta}^{1+\delta} \int_{-1+\delta}^{1+\delta} \rho(x,y) \,dx\,dy,
\]
where for fixed $x$ and $y$,
\begin{align}
\rho(x,y) &= \P(BC > xy) + 2\P(x+y<0\text{ and } {-(x-y)^2/4} < BC < xy) \notag \\
&= \P(BC > xy) + 2 \P(-(x-y)^2/4 < BC < xy) \one{x+y<0}
\label{first formula for rho}
\end{align}
(here $\one\cdot$ denotes the indicator function of the indicated relation). To complete the proof it suffices to show that $\rho$ equals the function $\nu$ defined in equation~\eqref{nu definition}.

The probabilities appearing in equation~\eqref{first formula for rho} are effectively given by Lemma~\ref{product distribution}.  However, there is some case-checking involved in applying this lemma, since the value of, say, $\P(BC > xy) = 1-F_{BC}(xy)$ depends on whether $xy < -1$, $-1 \le xy \le 1$, or $xy < -1$.  We make some observations to reduce the number of cases we need to examine.

Note that $(x-y)^2/4$ is bounded between 0 and 1 for any $x,y \in [-1+\delta,1+\delta]$, so that $-(x-y)^2/4$ always lies in the interval $[-1,1]$ prescribed by Lemma~\ref{product distribution}.  From the identity $(x+y)^2 - (x-y)^2 = 4xy$ we see also that $xy \ge -(x-y)^2/4$.  Thus $xy$ is never lower than $-1$, and we need only consider whether $xy>1$ (in which case $F_{BC}(xy)=1$).  We therefore have
\[
\P(BC > xy) = 1 - F_{BC}(xy) = \tfrac{1}{2}(1-G(xy)) \one{xy < 1}
\]
and
\begin{align*}
2 \P(-(x-y)^2/4 < BC < xy) & = 2 F_{BC}(xy) - 2 F_{BC}(-(x-y)^2/4) \\
& = \one{xy>1} + G(xy)\one{xy<1} + G((x-y)^2/4).
\end{align*}
Inserting these two evaluations into the formula~\eqref{first formula for rho}, we obtain
\begin{multline*}
\rho(x,y) = \tfrac{1}{2}(1-G(xy)) \one{xy < 1} \\
{}+ \big( \one{xy>1} + G(xy)\one{xy<1} + G((x-y)^2/4) \big) \one{x+y<0}.
\end{multline*}
It can be verified that this last expression is indeed equal to the right-hand side of the definition~\eqref{nu definition} of~$\nu$.
\end{proof}

Since $W(\delta) = W(-\delta) = -\frac{d}{d\delta} F_W(-\delta)$, to finish the proof of Theorem~\ref {theorem: distribution of real eigenvalues} it therefore suffices to prove that $-\frac{d}{d\delta} F_W(-\delta)$ equals the formula given in equation~\eqref {definition of W(delta)}.

\section{The derivative of the distribution}

Proposition~\ref {F_W proposition} expresses $F_W(\delta)$ as an integral, of a function $\nu$ that is independent of $\delta$, over the square $S_\delta = [-1+\delta,1+\delta]^2$. Since the region $S_\delta$ varies continuously with $\delta$, we can compute the derivative $-\frac{d}{d\delta} F_W(-\delta)$ by an appropriate line integral around the boundary  of $S_\delta$. Indeed, by the fundamental theorem of calculus, we have
\begin{align}
-\frac{d}{d\delta} F_W(-\delta) &= -\frac{1}{4} \frac{d}{d\delta} \bigg( \int_{-1+\delta}^{1+\delta} \int_{-1+\delta}^{1+\delta} \nu(x,y) \,dx\,dy \bigg) \notag \\
&= -\frac{1}{4} \bigg( \int_{-1+\delta}^{1+\delta} \nu(1+\delta,y)\,dy - \int_{-1+\delta}^{1+\delta} \nu(-1+\delta,y) \,dy \notag \\
&\qquad\qquad{}+ \int_{-1+\delta}^{1+\delta} \nu(x,1+\delta)\,dx - \int_{-1+\delta}^{1+\delta} \nu(x,-1+\delta) \,dx \bigg)\notag \\
          &= \frac{1}{2} \int_{-1+\delta}^{1+\delta} \nu(x,-1+\delta) \, dx -  \frac{1}{2} \int_{-1+\delta}^{1+\delta}  \nu(x,1+\delta) \, dx,
\label{eqn: key expression for W}
\end{align}
where we have used the symmetry $\nu(x,y) = \nu(y,x)$ to reduce the integral to just the top and bottom edges of $S_\delta$ (where $y=1+\delta$ and $y=-1+\delta$, respectively).

\begin{figure}[b]
\centering
\includegraphics[width= 1.75in]{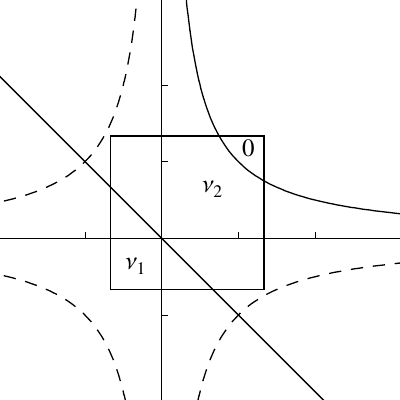}\hfill
\includegraphics[width= 1.75in]{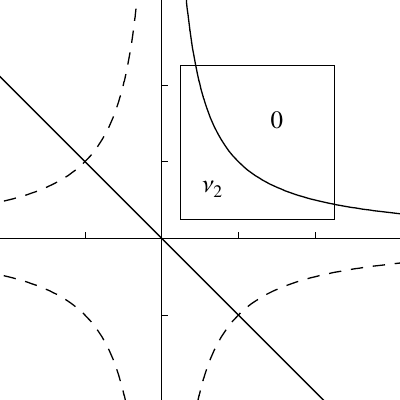}\hfill
\includegraphics[width= 1.75in]{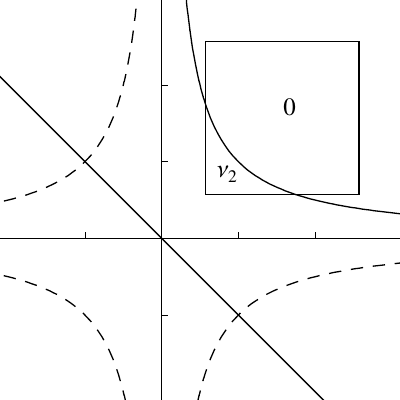}
\caption{The three cases: $0\le\delta\le1$, and $1\le\delta\le\sqrt2$, and $\sqrt2\le\delta\le2$}
\label{squares}
\end{figure}

The evaluation of~\eqref{eqn: key expression for W} divides into three cases depending on the behaviour of the indicator functions $\one{x+y<0}$ and $\one{xy<1}$ on the boundary of $S_\delta$ (see Figure~\ref{squares}).

\smallskip
{\em$\bullet$ Case 1: $0\le\delta\le 1$.}  For this range of $\delta$, the line $x+y=0$ intersects the bottom edge of $S_\delta$ at $x=1-\delta$, while the hyperbola $xy=1$ intersects the top edge at $x=(1+\delta)^{-1}$.
%Along the bottom edge of $S_\delta$, we have $\nu \equiv \nu_1$ for $x < 1-\delta$ and $\nu \equiv \nu_2$ for $x > 1-\delta$).  Along the top edge, we have $\nu \equiv \nu_2$ for $x < (1+\delta)^{-1}$ and $\nu \equiv 0$ otherwise.
Thus by the definition of $\nu$, equation~\eqref {eqn: key expression for W} becomes
\[
-\frac{d}{d\delta} F_W(-\delta) = \frac{1}{2} \bigg( \int_{-1+\delta}^{1-\delta} \nu_1(x,-1+\delta) \, dx + \int_ {1-\delta}^{1+\delta} \nu_2(x,-1+\delta) \, dx - \int_{-1+\delta}^{(1+\delta)^{-1}}  \nu_2(x,1+\delta) \, dx \bigg).
\]

The following elementary antiderivatives, which are readily obtained by substitution and integration by parts, follow for any fixed nonzero real number~$y$ from the definitions~\eqref{G definition},~\eqref{nu1 definition}, and~\eqref{nu2 definition} of $G$, $\nu_1$, and $\nu_2$:
\begin{align}
\int \nu_1(x,y)\, dx &= \tfrac{1}{2}x + \tfrac{1}{8}{x^2 y} (3 - 2\log|xy|) + \tfrac{1}{36} (x-y)^3(5 - 6 \log |(x-y)/2|), \notag \\
\int \nu_2(x,y)\, dx &= \tfrac{1}{2}x - \tfrac{1}{8}{x^2 y} (3 - 2\log|xy|).
\label{nu2 antiderivative}
\end{align}
%\begin{align}
%\int G(xy) \,dx &= \frac{x^2 y}{4}\big(3-2\log|xy|\big), 
%\label{antiderivatives1} \\
%\int G((x-y)^2/4) \, dx &= \frac{(x-y)^3}{36} \big(5 - 6 \log|(x-y)/2|\big).
%\label{antiderivatives2}
%\end{align}

Therefore in this case
\begin{align*}
-\frac{d}{d\delta} F_W(-\delta) &= \frac{1}{2} \bigg( \Big( \tfrac{1}{2}x + \tfrac{1}{8}{x^2(-1+\delta)} (3 - 2\log|x(-1+\delta)|) \\
&\hskip1in {}+ \tfrac{1}{36} (x+1-\delta) ^3(5 - 6 \log |(x+1-\delta)/2|) \Big) \Big|_{x=-1+\delta}^{1-\delta} \\
&\qquad{}+  \big( \tfrac{1}{2}x - \tfrac{1}{8}{x^2 (-1+\delta)} (3 - 2\log|x(-1+\delta)|) \big) \Big|_{x=1-\delta}^{1+\delta} \\
&\qquad{}-  \big( \tfrac{1}{2}x - \tfrac{1}{8}{x^2(1+\delta)} (3 - 2\log|x(1+\delta)|) \big) \Big|_{x=-1+\delta}^{(1+\delta)^{-1}} \bigg)
\\
&= \frac{80 + 20\delta + 90\delta^2 + 52\delta^3 - 107\delta^4}{144(1+\delta)} - \frac{(5-7\delta+8\delta^2)(1-\delta)}{12} \log(1-\delta) \\ &  \qquad - \frac{\delta(1-\delta^2)}{4} \log(1+\delta)
\end{align*}
(after some algebraic simplification), which verifies the first case of Theorem~\ref{theorem: distribution of real eigenvalues}. (Note that the integrands really are continuous, despite terms that look like $\log0$, because the function $G$ is continuous at~0; hence evaluating the integrals by antiderivatives is valid.)

\smallskip
{\em$\bullet$ Case 2: $1\le\delta\le \sqrt{2}$.}  Now, the line $x+y=0$ does not intersect $S_\delta$, while the hyperbola $xy=1$ intersects the top edge at $x=(1+\delta)^{-1}$. Thus by the definition of $\nu$ and the antiderivative~\eqref{nu2 antiderivative} of $\nu_2$, equation~\eqref {eqn: key expression for W} becomes
\begin{align*}
-\frac{d}{d\delta} F_W(-\delta) &= \frac{1}{2} \bigg( \int_{-1+\delta}^{1+\delta} \nu_2(x,-1+\delta) \, dx - \int_{-1+\delta}^{(1+\delta)^{-1}}  \nu_2(x,1+\delta) \, dx \bigg) \\
&= \frac{1}{2} \bigg(  \big( \tfrac{1}{2}x - \tfrac{1}{8}{x^2 (-1+\delta)} (3 - 2\log|x(-1+\delta)|) \big) \bigg|_{x=-1+\delta}^{1+\delta} \\
&\qquad{}-  \big( \tfrac{1}{2}x - \tfrac{1}{8}{x^2(1+\delta)} (3 - 2\log|x(1+\delta)|) \big) \bigg|_{x=-1+\delta}^{(1+\delta)^{-1}} \bigg)
\\
&= \frac{\delta(20+10\delta-12\delta^2-3\delta^3)}{16(1+\delta)} + \frac{(3\delta-1)(\delta-1)}{4} \log(\delta-1)
 + \frac{\delta (\delta^2-1)}{4} \log(\delta+1),
\end{align*}
which verifies the second case of Theorem~\ref{theorem: distribution of real eigenvalues}.

\smallskip
{\em$\bullet$ Case 3: $\sqrt{2} < \delta\le 2$.}  As before, the line $x+y=0$ does not intersect $S_\delta$, while the hyperbola $xy=1$ intersects the bottom edge at $x=(\delta-1)^{-1}$. Thus by the definition of $\nu$ and the antiderivative~\eqref{nu2 antiderivative} of $\nu_2$, equation~\eqref {eqn: key expression for W} becomes
\begin{multline*}
-\frac{d}{d\delta} F_W(-\delta) = \frac{1}{2} \int_{-1+\delta}^ {(-1+\delta)^{-1}} \nu_2(x,-1+\delta) \, dx = \frac{1}{2}  \big( \tfrac{1}{2}x - \tfrac{1}{8}{x^2 (1+\delta)} (3 - 2\log|x(1+\delta)|) \big) \bigg|_{x=-1+\delta}^{(-1+\delta)^{-1}} \\
{}= \frac{\delta(\delta-2)(2-6\delta+3\delta^2)}{16(\delta-1)} - \frac{(\delta-1)^3}{4} \log(\delta-1),
\end{multline*}
which verifies the third case of Theorem~\ref{theorem: distribution of real eigenvalues}.

\smallskip
Since the last case of Theorem~\ref{theorem: distribution of real eigenvalues} is a consequence of Lemma~\ref {lemma: baby Gershgorin}, the proof of the theorem is complete.

\begin{remark}
One could also use the same method to extract the individual distributions of the greater and lesser eigenvalues of $M$: for instance, eliminating the factor of 2 from equation~\eqref{first formula for rho} would yield an expression for the distribution of just the lesser eigenvalue of~$M$.
\end{remark}

\end{document}